\documentclass{article}
\usepackage{style}
\title{Noncommutative Bell polynomials, quasideterminants and incidence Hopf algebras}

\author{
Kurusch Ebrahimi-Fard\footnote{\small ICMAT, Calle Nicolás Cabrera 13-15, Campus de Cantoblanco, UAM, Madrid, Spain. On leave from Université de Haute Alsace, 18 Rue des Frères Lumière, 68093 Mulhouse Cedex, France.
{\small{kurusch@icmat.es}}}
\and
Alexander Lundervold\footnote{INRIA Bordeaux, 200 Avenue de la Vieille Tour, 33405 Talence Cedex, France.
{\small{alexander.lundervold@gmail.com}}}
\and
Dominique Manchon\footnote{Université Blaise Pascal, 24 Avenue des Landais
Les Cézeaux, BP 80026, F63177 Aubière Cedex, France.
{\small{manchon@math.univ-bpclermont.fr}}}}

\begin{document}

\maketitle

\tableofcontents

\begin{abstract}
Bell polynomials appear in several combinatorial constructions throughout mathematics. Perhaps most naturally in the combinatorics of set partitions, but also when studying compositions of diffeomorphisms on vector spaces and manifolds, and in the study of cumulants and moments in probability theory. We construct commutative and noncommutative Bell polynomials and explain how they give rise to Faà di Bruno Hopf algebras. We use the language of incidence Hopf algebras, and along the way provide a new description of antipodes in noncommutative incidence Hopf algebras, involving quasideterminants. We also discuss Möbius inversion in certain Hopf algebras built from Bell polynomials.
\end{abstract}


\let\thefootnote\relax

\section{Introduction}

In \cite{bell1927pp, bell1934ep} E.T. Bell introduced a family of commutative polynomials related to set partitions, named Bell polynomials by Riordan \cite{riordan1958ait}. Noncommutative versions were introduced by Schimming in \cite{schimming1996nbp} and by Munthe-Kaas in \cite{munthe-kaas1995lbt}, the latter in the setting of numerical integration on manifolds. They also appeared recently in relation to quasi-symmetric functions in \cite[Section 4.5]{novelli2013bsb}. In this work we study various descriptions of commutative and noncommutative Bell polynomials, both recursive and explicit, via partitions of sets, determinants and quasideterminants. We also investigate the link to \emph{Faà di Bruno formulas} describing compositions of diffeomorphisms. The classical Faà di Bruno formula expresses derivatives of compositions of functions on $\mathbb{R}$ as
\begin{align}
  \frac{d^n}{dx^n} f(g(x)) = \sum_{k=0}^n f^{(k)}(g(x)) B_{n,k}\left(g'(x), g''(x), \dots, g^{(n-k+1)}(x)\right),
\end{align}
where $B_{n,k}$ are the commutative (partial) Bell polynomials. As explained in Johnson's fascinating historical account \cite{johnson2002tch}, what is now called the Faà di Bruno formula was actually discovered and studied many times prior to Faà di Bruno's work. However, he did obtain a new determinantal formulation, related to a general determinantal formula for commutative Bell polynomials. In the present paper we generalize his result by obtaining a \emph{quasi-}determinantal formula for noncommutative Bell polynomials (Section \ref{Sect:quasideterminants}).

Trying to capture the Faà di Bruno formula algebraically leads to a Hopf algebra, called the Faà di Bruno Hopf algebra. The more general setting of diffeomorphisms on manifolds leads to the \emph{Dynkin Faà di Bruno Hopf algebra}. The main result linking diffeomorphisms to noncommutative Bell polynomials is a formula expressing the pullback of a function $\psi$ along a (time-dependent) vector field $F_t$:
\begin{align}
	\frac{d^n}{dt^n} \Phi_{t,F_t}^* \psi = B_n(F_t)[\psi],
\end{align}
where $B_n$ is a noncommutative Bell polynomial (Section \ref{Sect:FdBManifolds}).

In Section \ref{Sect:incidenceHopf} we formulate these Hopf algebras as \emph{incidence Hopf algebras}, in which the Bell polynomials are the so-called \emph{rank polynomials} of the underlying posets. In \cite{schmitt1994iha} the antipodes in a class of commutative incidence Hopf algebras were described as determinants of certain polynomials. We extend this result to noncommutative incidence Hopf algebras using quasideterminants. 

We end with a short section (Section \ref{Sect:mobius}) formulating a theory of Möbius inversion for Hopf algebras built from Bell polynomials, which allows us to express the indeterminats in terms of the Bell polynomials.

\section{Constructions}
\label{Sect:constrBell}

We present ways to construct the Bell polynomials, both in commuting and noncommuting variables.

\subsection{Recursive descriptions}
\label{Sect:noncommBell}

Bell polynomials have several convenient recursive descriptions. One of the advantages of these is that they are valid regardless of whether the underlying algebraic setting is commutative or not. Explicit formulas can be found in Section \ref{Sect:bellAndPart}.

\subsubsection{Basic recursive descriptions}
\label{Sect:basicRecursive}

Consider an alphabet $\{d_i\}$, where the letters are indexed by natural numbers, and graded by $|d_i| = i$. The space of \emph{words} in this alphabet, $\mathcal{D} = \K\langle \{d_i\} \rangle$, comes equipped with the concatenation operation, and is graded by $|d_{j_1} \cdots d_{j_k}| = |d_{j_1}| + \cdots + |d_{j_k}| = j_1 + \cdots + j_k$, extended linearly. We also equip $\mathcal{D}$ with a linear derivation $\partial: \mathcal{D} \rightarrow \mathcal{D}$ defined as 
\begin{align}
	\partial(d_i) = d_{i+1},
\end{align}
and extended to words by the Leibniz rule. Write $\mathbb{I}$ for the empty word in $\mathcal{D}$. Iteratively multiplying with an element from the left plus a derivation generates the Bell polynomials in $\mathcal{D}$:

\begin{definition}
\label{def:recursiveBell}
The \emph{Bell polynomials} are defined recursively by
\begin{align}
	B_0 &= \mathbb{I} \\
	B_n &= (d_1 + \partial) B_{n-1}, \quad n > 0.
\end{align}
\end{definition}

Whether these are the commutative or noncommutative Bell polynomials is determined by whether the $d_i$ commute. Note that since $\partial(\mathbb{I}) = 0$ we can write
\begin{align}
	B_n = (d_1 + \partial)^n \mathbb{I}.
\end{align}
A simple induction gives the following alternative description:

\begin{proposition}
\label{prop:recursionBell}
The Bell polynomials (commutative or noncommutative) satisfy the recursion
\begin{align}
	B_{n+1} &= \sum_{k=0}^n {n \choose k} B_{n-k} d_{k+1}\\
	B_0 &= \mathbb{I}.
\end{align} 
\end{proposition}

\paragraph{Examples.} Here are the first few noncommutative Bell polynomials. The number of terms grows exponentially, with $2^{n-1}$ terms in $B_n$.
\begin{align}
\begin{array}{l|c}
B_0 & 1 \\[4pt] \hline \\[-4pt]
B_1 & d_1\\[4pt] \hline \\[-4pt]
B_2 & d_1^2 + d_2\\[4pt] \hline \\[-4pt]
B_3 & d_1^3 + d_2 d_1 + 2d_1 d_2 + d_3\\[4pt] \hline \\[-4pt]
B_4 & d_1^4 + 3d_1^2 d_2 + 3d_2^2 + d_3 d_1 + d_2 d_1^2 + 2d_1d_2d_1 + 3d_1d_3 + d_4\\[4pt] \hline \\[-4pt]
B_5 & d_1^5 + 6d_1^2d_3 + 6d_2d_3 + 4d_3d_2 + 4d_1^3d_2 + 4d_2d_1d_2 + 8d_1d_2^2 + \\[4pt] 
& d_4d_1 + 3d_1^2d_2d_1 + 3d_2^2 d_1 + d_3d_1^2 + d_2d_1^3 + 2d_1d_2d_1^2 + 3d_1d_3d_1 + 4d_1d_4 + d_5
\end{array}
\end{align}
The coefficients in these polynomials are intriguing, and will be described in detail in Section \ref{Sect:bellAndPart}. 

The \emph{grade} of a word $\omega$ in the polynomial $B_n$ is $|\omega| = n$. We can also consider the \emph{length} of the words, written $\#(\omega)$. This leads to the \emph{partial Bell polynomial} $B_{n,k}$, which is the part of $B_n$ consisting of words of length $k$. For example, 
\begin{align}
	B_{3,2} = d_2 d_1 + 2 d_1 d_2.
\end{align}
The scaled scaled Bell polynomials defined as 
\begin{align}
\label{Qpolys}
	Q_n = \sum_k Q_{n,k}, \quad \text{where  } Q_{n,k} = \frac{1}{n!} B_{n,k}(1!d_1, 2!d_2, 3!d_3, \dots). 
\end{align}
will be of interest later. 
\begin{align}
	Q_2 &= \frac{1}{2}\big(d_1^2 + 2d_2\big)\\
	Q_3 &= \frac{1}{6}\big(d_1^3 + 2d_2d_1 + 4d_1d_2 + 6d_3\big).
\end{align}

\subsubsection{Description in terms of trees}

Using combinatorial trees one can give another simple recursive description for the Bell polynomials. A \emph{rooted tree} is a finite simple graph without cycles, with a distinguished vertex called the root. A \emph{rooted forest} is a graph whose connected components are rooted trees. Write $T$ for the set of rooted trees, and $\T = \K\langle T \rangle$. The operation $B^+$ from forests to trees adds a common root to a forest. Any rooted tree can be written as $B^+$ of a forest, $t = B^+(t_1, t_2, \dots, t_n)$. The \emph{left Butcher product} $\lbutch: \T \otimes \T \rightarrow \T$ is defined as 
\begin{align}
	s \lbutch t = B^+(s, t_1, t_2, \dots, t_n).
\end{align}
The operation $\tlgraft: \T\otimes \T \rightarrow \T$ is given by grafting a tree to the leaves of another tree. For example,
\begin{align}
	\aabb\ \tlgraft\ \aaabbb =  \aaaaabbbbb,\quad
	\aabb\ \tlgraft \aababb = \aaaabbbabb + \aabaaabbbb.
\end{align}
We associate a letter $d_i$ to the ladder tree with $i$ edges. E.g. 
\begin{align}
	d_0 \sim \ab, \quad d_1 \sim \aabb, \quad d_2 \sim \aaabbb.
\end{align}
Writing $\hat{d}_l$ for the ladder tree with $l$ edges, the concatenation operation on letters $d_i d_j$ corresponds to the left Butcher product $\hat{d}_{i-1} \lbutch \hat{d}_j$. The derivation operation $\partial(d_i) = d_{i+1}$ corresponds to left grafting on the leaf of $\hat{d}_i$, that is, $\partial = \ab \tlgraft \_$. The next result then follows from the definition of Bell polynomials.

\begin{proposition} 
\label{prop:BellTrees}
The Bell polynomials can be generated recursively by
\begin{align}
	\hat{B}_n = \ab (\lbutch  + \tlgraft) \hat{B}_{n-1}
\end{align} 
\end{proposition}
The first four noncommutative Bell polynomials correspond to
\begin{align}
& \hat{B}_0 = \ab, \quad \hat{B}_1 = \aabb, \quad \hat{B}_2 = \aababb + \aaabbb,\\
& \hat{B}_3 = \aabababb + \aaabbabb + 2 \aabaabbb + \aaaabbbb, \\
& \hat{B}_4 = \aababababb + 3\aababaabbb + 3\aaabbaabbb + \aaaabbbabb + \aaabbababb + 2\aabaabbabb + 3\aabaaabbbb + \aaaaabbbbb
\end{align}
Note that if the trees are nonplanar (i.e. the order of the branches is insignificant) we obtain the commutative Bell polynomials. If planar, the noncommutative Bell polynomials. 

\begin{remark}
There is a link between the Bell polynomials and the so-called \emph{natural growth operator} on trees. This is also related to \emph{Lie--Butcher series} and the flow of differential equations on (homogeneous) manifolds. See \cite{lundervold2011hao}. The link goes via the so-called \emph{Grossman--Larson product}, and is currently being investigated.
\end{remark}

\subsection{Explicit formulas}
\label{Sect:bellAndPart}

Bell polynomials can be given several explicit descriptions. We begin with determinantal descriptions, then descriptions via partitions, and finally a description related to the Dynkin idempotent.

\subsubsection{Determinants and quasideterminants}
\label{Sect:quasideterminants}

It is well known that the classical Bell polynomials of \cite{bell1927pp} can be defined in terms of determinants (see e.g. \cite{schimming1999dpe}). For example:
\begin{align}
\left|\begin{array}{@{}*{3}{r}@{}}
d_1 &  {3-2 \choose 1} d_2& {3-1 \choose 2 }d_3  \\ \\
-1 & d_1 & {3-1 \choose 1} d_2 \\ \\
0 & -1  & d_1
\end{array}\right|
&=
\left|\begin{array}{@{}*{3}{r}@{}}
d_1 & d_2 & d_3  \\  \\
-1 & d_1 & 2d_2  \\  \\
0 &  -1  & d_1 
\end{array}\right|
&\\
&= d_1^3+3d_1 d_2 + d_3\\
&= B_3(d_1,d_2,d_3).
\end{align}
The result can be found indirectly already in Fa\`a di Bruno's work (\cite{faadibruno1855ssd, faadibruno1857nsu}) from the 1850s.

\begin{theorem}
The commutative Bell polynomial $B_n$ can be written as
\begin{align}
	B_n(d_1, \dots, d_n) = |\mathbf{B}_n|,
\end{align}
where $\mathbf{B}_n$ is the $n \times n$ matrix whose second diagonal elements are $-1$, the lower elements all zero, and the remaining entries are given by 
\begin{align}
	(\mathbf{B}_n)_{ij} = {n - (n - j +1) \choose j-i} d_{j-i+1}, \qquad \text{for  } i \leq j.
\end{align}
\end{theorem}

It turns out that the noncommutative Bell polynomials have a rather similar description, in terms of a noncommutative analog of the determinant: the \emph{quasideterminants} of Gelfand and Retakh (\cite{gelfand1991dom}, see also \cite{gelfand2005q})\footnote{The link between noncommutative Bell polynomials and quasideterminants was first remarked upon in \cite{lundervold2011oas}}. Note that a detailed account of the theory of quasideterminants is beyond the scope of this paper. We content ourselves with recalling the definition and some simple consequences. For more details the reader may consult the references given above.    

Write $A^{pq}$ for the matrix obtained by deleting the $p$th row and $q$th column of square matrix $A$.

\begin{definition}
\label{def:qdet} 
The quasideterminant $|A|_{pq}$ of order $pq$ of the matrix $A$ consisting of $n^2$ noncommuting indeterminates $a_{ij}$, $1\leq i,j \leq n$, is defined by
\begin{align}
	|A|_{pq} = a_{pq} - \sum_{i\neq p, j\neq q} a_{pj} (|A^{pq}|_{ij})^{-1} a_{iq}.
\end{align}
\end{definition}
For example, we have 
\begin{align}
\left|\begin{array}{ccc}
x_1 & x_2 & \text{\circled{$x_3$}}  \\  \\
-1 & x_1 & 2x_2  \\  \\
0 &-1  & x_1  
\end{array}\right| = x_1^3+2x_1 x_2 + x_2 x_1 + x_3,
\end{align}
where we circle the element corresponding to the quasideterminant.  

Quasideterminants also satisfy a slightly simpler looking formula
\begin{align}
	|A|_{pq} = a_{pq} - \sum_{i\neq p, j\neq q} a_{pj} ((A^{pq})^{-1})_{ij} a_{iq},
\end{align}
which yields a nice pictorial description of quasideterminants:\footnote{Picture source: Wikipedia / Aaron Lauve}
\begin{center}
	\includegraphics[width=0.8\textwidth]{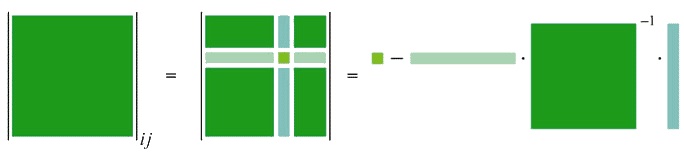}
\end{center}
As the entries of $A^{pq}$ are noncommutative, it is not obvious how one should interpret the inverse $(A^{pq})^{-1}$. It involves quasideterminants of submatrices, similar to the commutative calculation of the inverse in terms of cofactors. See \cite{gelfand2005q}.

\begin{remark}\label{commutativeQuasi} If the elements of the matrix commute then the quasideterminant is equal to the classical determinant divided by a minor:
\begin{align}
	|A|_{pq} = (-1)^{p+q}\frac{\det A}{\det A^{pq}}.
\end{align}
\end{remark}

\begin{remark}
\label{commQuasiPoly} 
If all the elements commute, then the minor $\det A^{1n}$ of the above matrix is $\det A^{1n} = (-1)^{n-1}$, and, by Remark \ref{commutativeQuasi}, the quasideterminant is equal to the classical determinant.
\end{remark}

In the above example the quasideterminant was a polynomial (in fact, the third noncommutative Bell polynomial, $B_3$), which is of course not true in general. However, certain quasideterminants are guaranteed to be polynomials in their entries, and they can be described explicitly. This description simplifies many of our calculations.

\begin{proposition}[{\cite[Proposition 1.2.9]{gelfand2005q}}]\label{quasiPoly}
The following quasideterminant is polynomial in its entries and has a nonrecursive description.
\begin{align}
P(n) &
= \left|\begin{array}{ccccc}
a_{11} & a_{12} & a_{13} & \cdots  & \text{\circled{$a_{1n}$}}\\  \\
-1 & a_{22} & a_{23} & \cdots & a_{2n} \\ \\
0 & -1 & a_{33} & \cdots & a_{3n}\\ \\
 & \cdots &  &  \\  \\
0 & \cdots & 0  & -1 & a_{nn} 
\end{array}\right|
\\ &\quad = a_{1n} + \sum_{1 \leq j_1 < j_2 < \cdots < j_k < n} a_{1j_1} a_{j_1+1,j_2} a_{j_2+1,j_3} \cdots a_{j_k+1,n}.
\end{align}
\end{proposition}
We get
\begin{align}
P(3) &= a_{13} + a_{11}a_{23} + a_{12}a_{33} + a_{11}a_{22}a_{33}\\
P(4) &= a_{14} + a_{11}a_{24} + a_{12}a_{34} + a_{13}a_{44} + a_{11}a_{22}a_{34} + a_{11}a_{23}a_{44}\\
&\quad  + a_{12}a_{33}a_{44} + a_{11}a_{22}a_{33}a_{44}.
\end{align}

Note that $P(4)$ can be expanded as follows:
\begin{align}
P(4) &= a_{14} |I| + a_{13} \left|\text{\circled{$a_{44}$}}\right| + a_{12} \left|\begin{array}{cc} a_{33} & \text{\circled{$a_{34}$}} \\ -1 & a_{44} \end{array}\right| + a_{11} \left|\begin{array}{ccc} a_{22} & a_{23} & \text{\circled{$a_{24}$}} \\ -1 & a_{33} & a_{34} \\ 0 & -1 & a_{44} \end{array}\right|. 
\end{align}
Write $M_n$ for the $n \times n$ matrix which figures in $P(n) = |M_n|_{1n}$. Let $M_k$ be the matrix recursively defined as $M_{k} = (M_{k+1})^{11}$ for $0< k < n$, i.e. by deleting the first row and first column of $M_{k+1}$, and set $M_0 = I$, the $1 \times 1$ identity matrix. The above formula then reads
\begin{align}
	P(4) = a_{14}|M_0| + a_{13}|M_1|_{11} + a_{12}|M_2|_{12} + a_{11}|M_3|_{13}.
\end{align}
In general, we have the following result, which will be of importance later. 

\begin{proposition}\label{quasidetExpansion}
The quasideterminant $P(n)$ of Proposition \ref{quasiPoly} can be written as an expansion over quasideterminants of submatrices
\begin{align}
	P(n) &= \sum_{k=0}^{n-1} a_{1,n-k} |M_k|_{1k},
\end{align}
where the matrices $M_k$, $k=0,\ldots, n-1$, are as above, defined iteratively from the matrix $M_n$. In addition, the following recursion holds
\begin{align}
	P(n) &= \sum_{k=1}^n P(k-1) a_{kn},
\end{align}
where we set $P(1) = I$. 
\end{proposition}

\begin{proof}
Both formulas follow from the expansion in Proposition \ref{quasiPoly}:
\begin{align}
P(n) = a_{1n} &+ a_{11} \sum_{2 \leq j_2 < j_3 < \cdots < j_k < n} a_{2j_2} a_{j_2+1,j_3} \cdots a_{j_k + 1,n} \\
&+ a_{12} \sum_{3 \leq j_3 < j_4 < \cdots < j_k < n} a_{3j_1} a_{j_1+1,j_2} \cdots a_{j_k + 1,n} \\&  \vdots \\ 
&+ a_{1,n-1} \sum_{n-1 \leq j_k < n} a_{j_k + 1,n} \\
&+ a_{1n}
\end{align}
and
\begin{align}
P(n) = a_{1n} &+ a_{11} a_{2n} \\
&+ \left(a_{12} + \sum_{1 \leq j_1 < j_2 < 3} a_{1j_1}a_{j_1+1,j_2} a_{j_2 +1, 2}\right) a_{3n}\\
&+ \left(a_{13} + \sum_{1 \leq j_1 < j_2 < j_3 < 4} a_{1j_1}a_{j_1+1,j_2} a_{j_2+1, j_3} a_{j_3 +1, 3}\right) a_{4n}\\
&\vdots \\
&+  \left(a_{1,n-1} + \sum_{1 \leq j_1 < j_2 < \cdots < j_{k-1} < n} a_{1j_1} a_{j_1+1,j_2} \cdots a_{j_{k-1} + 1,n-1}\right)a_{nn}
\end{align}
\end{proof}

Using Proposition \ref{quasiPoly} we can easily calculate the following $4\times 4$ quasideterminant:
\begin{align}
&\left|\begin{array}{cccc}
x_1 & x_2 & x_3  & \text{\circled{$x_4$}}\\  \\
-1 & x_1 & 2x_2 & 3 x_3 \\ \\
0 & -1 & x_1 & 3 x_2 \\ \\
0 & 0 & -1 & x_1 
\end{array}\right| \\
&= \quad x_4 + 3x_1x_3 + 3x_2^2 + x_3x_1 + 3x_1^2 x_2 + 2x_1 x_2 x_1 + x_2x_1^2 + x_1^4.
\end{align}

This is the fourth noncommutative Bell polynomial, $B_4$, which is no accident: the noncommutative Bell polynomials are given by the quasideterminant in Proposition \ref{quasiPoly}. The entries of the matrix $M_n$ are $a_{ii} = d_1$ and
\begin{align}a_{ij} = {n - (n - j+1) \choose j-i} d_{j-i+1} = {j-1 \choose j-i} d_{j-i+1} = {j-1 \choose i-1} d_{j-i+1},\end{align}
for $1 \leq i < j \leq n$.

\begin{theorem}
The noncommutative Bell polynomial $B_n$ can be written as
\begin{align}
	B_n(d_1, \dots, d_n) = \qdet{\mathbf{B}_n}_{1n},
\end{align}
where $\mathbf{B}_n$ is the $n \times n$ matrix whose second diagonal elements are $-1$, the lower elements are all zero, and the elements on and above the diagonal are 
\begin{align}
	(\mathbf{B}_n)_{ij} = {n - (n-j+1) \choose j-i} d_{j-i+1}, \qquad \text{for  } i \leq j.
\end{align}
We set $\mathbf{B}_0 = \mathbb{I}$. 
\end{theorem}

\begin{proof}
The coefficients in the last column of $\mathbf{B}_{n+1}$ are ${n \choose n-i}$, so by expanding along this column using the second formula of Proposition \ref{quasidetExpansion} we find that the quasideterminant satisfies the recursion
\begin{align}
	\qdet{\mathbf{B}_{n+1}}_{1n} = \sum_{k=0}^n {n \choose k} \qdet{\mathbf{B}_{n-k}}_{1n} d_{k+1}, \qquad \mathbf{B}_0 = \mathbb{I},
\end{align}
and therefore equals the noncommutative Bell polynomial.
\end{proof}

As an example, 
\begin{align}
B_6 = 
\left|\begin{array}{cccccc}
{6 - 6 \choose 0} d_1 & {6-5 \choose 1} d_2 & {6-4 \choose 2} d_3  & {6-3 \choose 3} d_4 & {6-2 \choose 4} d_5 & \text{\circled{${6-1 \choose 5}d_6$}}\\  \\
-1 & {6 - 5 \choose 0} d_1 & {6-4 \choose 1} d_2 & {6-3 \choose 2} d_3 & {6-2 \choose 3} d_4 & {6-1 \choose 4} d_5 \\ \\
0 & -1 & {6 - 4 \choose 0} d_1 & {6-3 \choose 1} d_2 & {6-2 \choose 2} d_3 & {6-1 \choose 3} d_4 \\ \\
0 & 0 & -1 & {6 - 3 \choose 0} d_1 & {6-2 \choose 1} d_2 & {6-1 \choose 2} d_3 \\ \\
0 & 0 & 0 & -1 & {6 - 2 \choose 0}d_1 & {6-1 \choose 1} d_2 \\ \\
0 & 0 & 0 & 0 & -1 & {6 - 1 \choose 0}d_1 \\ \\
\end{array}\right| \\
\end{align}

\subsubsection{Partitions}

Based on the recursive description of commutative Bell polynomials in Proposition \ref{prop:recursionBell} one can check that 
\begin{align}
	B_n = \sum_{\alpha_1 + 2\alpha_2 + \cdots + n\alpha_n = n} \frac{n!}{\alpha_1! \alpha_2! \cdots \alpha_n!} \left(\frac{d_1}{1!}\right)^{\alpha_1} \left(\frac{d_2}{2!}\right)^{\alpha_2} \cdots \left(\frac{d_n}{n!}\right)^{\alpha_n},
\end{align}
and therefore that the commutative partial Bell polynomials can be written as
\begin{align}
	B_{n,k} = \sum_{ \alpha_1 + \alpha_2 + \cdots + \alpha_n = k \atop  \alpha_1 + 2\alpha_2 + \cdots + n\alpha_n = n } \frac{n!}{\alpha_1! \alpha_2! \cdots \alpha_n!} \left(\frac{d_1}{1!}\right)^{\alpha_1} \left(\frac{d_2}{2!}\right)^{\alpha_2} \cdots \left(\frac{d_n}{n!}\right)^{\alpha_n}.
	\label{explicitCPBell}
\end{align}

This can also be shown using the description of Bell polynomials via exponential generating series:
\begin{align}\label{genSeriesComBell}
  B(t):=1+ \sum_{n>0} B_n\frac{t^n}{n!} = \exp\bigg( \sum_{m>0} d_m \frac{t^m}{m!}\bigg),
\end{align}
see e.g. \cite{bell1934ep}.

Alternatively, we can obtain the formula by linking Bell polynomials to set partitions, an approach that will provide us with similar explicit descriptions for noncommutative Bell polynomials.

A \emph{partition} $P$ of a set $[n] = \{1,2,\dots,n\}$ of $n$ elements into $k$ blocks is an unordered collection of $k$ non-empty disjoint sets $\{P_1, P_2, \dots, P_k\}$, called the \emph{blocks} or \emph{parts} of the partition, whose union is $[n]$. The \emph{size} $|P_i|$ of a block is the number of elements in $P_i$. The set of all partitions of $[n]$ is written as $\mathcal{P}_n$, and the set of all partitions of $[n]$ into $k$ blocks as $\mathcal{P}_{n,k}$. Note that if we write $|P|_m$ for the number of blocks of size $m$ in $P$, we have
\begin{align}
	\sum_{m=1}^n |P|_m = k, \qquad \sum_{m=1}^n m|P|_m = n.
\end{align}
The number of ways to partition $[n]$ into prescribed blocks of sizes $j_1, j_2, \dots, j_k$ is
\begin{align}
M(n;j_1, \dots, j_k) &= {n \choose j_1} {n-j_1 \choose j_2} \cdots {n-(j_1 + \cdots + j_{k-1}) \choose j_k},\\
                    &= {n \choose j_1, j_2, \dots, j_k}.
\end{align}
We first choose $j_1$ elements to put in the first block, then $j_2$ from the remaining elements to put in the second block, etc. This is the well-known multinomial coefficient. We will need to count other types of partitions later. For more on the combinatorial study of set partitions consult one of the many excellent sources on this topic, e.g. \cite{stanley2011ec1}.

\paragraph{Commutative Bell polynomials.} 

Commutative Bell polynomials can be described in a straightforward way by summing over all partitions. Let $P \in \mathcal{P}_n$ with $k$ blocks $\{P_1,\ldots, P_k\}$. Define the map $d(P):=\prod_{j=1}^k d_{|P_j|}$ taking value in the commutative algebra $\mathcal{D}=\K\langle \{d_i\} \rangle$. Observe that $d_n=d(1_n)$, where $1_n \in  \mathcal{P}_n$ is the partition of $[n]$ into a single block of size $n$. The commutative Bell polynomial is then
\begin{align}
\label{SumPartComBell}
	B_n=\sum_{P \in \mathcal{P}_n} d(P). 
\end{align} 
A natural question is whether this sum can be inverted, i.e., whether $d_n$ can be written in terms of $B_i$, $i=1,\ldots,n$. The well-known answer to this is given by using M\"obius inversion on the lattice $\mathcal{P}_n$  \cite{rota1964otf}. Section \ref{Sect:mobius} will expand upon this from an Hopf algebraic point of view. 

We can give a more precise description of the sum in Equation \eqref{SumPartComBell}. The coefficients of a commutative partial Bell polynomial $B_{n,k}$ add up to the number of partitions of a set $\{1,\dots, n\}$ into $k$ blocks, i.e. they are given by the so-called \emph{Stirling numbers of the second kind}. For example, the coefficient in front of $d_1d_2$ in $B_3$ is $3$ because there are $3$ ways to partition $[3]$ into two blocks of sizes $|d_1| = 1$ and $|d_2| = 2$:
\begin{align}
	\quad 1 | 2 3, \quad 3 | 1 2, \quad 2 | 3 1.
\end{align}
In general,
\begin{align}
	B_{n,k}(1,1,\dots,1) = {n \brace k},
\end{align}
where ${n \brace k}$ are the Stirling numbers of the second kind. We think of the symbol $d_i$ as representing a block of size $i$, and write the commutative Bell polynomials as
\begin{align}
\label{partitionCBell}
	B_{n,k} = \sum_{|\omega| = n \atop \#(\omega) = k} M(n; j_1, j_2, \dots, j_k) \omega,
\end{align}
where $\omega= d_{j_1} \cdots d_{j_k}$ is a word in the commuting variables $d_i$, and $M(n; j_1, \dots, j_k)$ is the multinomial coefficient introduced above. Since the variables are commutative we can rewrite all words in the order of increasing number of elements in the blocks, i.e. write 
\begin{align}
	d_{j_1} \cdots d_{j_k} = d_1^{\alpha_1} \cdots d_n^{\alpha_n}.
\end{align}
If we now add up all the words, we are going to see each $M(n; \alpha_1, \dots, \alpha_n)$ times. The order within the blocks $d_i^{\alpha_i}$ does not matter, and we obtain
\begin{align}
&\sum_{|\omega| = n \atop \#\omega = k} M(n; j_1, j_2, \dots, j_k) \omega \\
&\quad = \sum_{\alpha_1 + \alpha_2 + \cdots + \alpha_n = k \atop \alpha_1 + 2\alpha_2 + \cdots + n\alpha_n = n} {n \choose \alpha_1, \alpha_2, \dots, \alpha_n} \left(\frac{d_1}{1!}\right)^{\alpha_1}  \left(\frac{d_2}{2!}\right)^{\alpha_2} \cdots \left(\frac{d_n}{n!}\right)^{\alpha_n},
\end{align}
i.e. Formula \eqref{explicitCPBell}.

\paragraph{Noncommutative Bell polynomials.} 

Formula \eqref{SumPartComBell} also gives a direct way to write down noncommutative Bell polynomials, based on imposing an order on the blocks of the partitions.

The coefficients of a noncommutative partial Bell polynomial $B_{n,k}$ count the number of partitions of $[n]$ into $k$ blocks, \emph{ordered by the maximum of each block}. For example, the coefficient of $d_2d_1d_2$ in $B_{5,3}$ is the number of ways to partition the set $\{1,2,3,4,5\}$ into three parts $P_1$, $P_2$, $P_3$ of sizes $|P_1|=2$, $|P_2|=1$ and $|P_3|=2$, such that $\max(P_1) < \max(P_2) < \max(P_3)$.
\begin{align}
	d_2 d_1 d_2: \quad 1 2 | 3 | 4 5, \quad 1 2 | 4 | 3 5, \quad 1 3 | 4 | 2 5, \quad 2 3 | 4 | 1 5.
\end{align}
To see this we will first count the number of such partitions, then relate it to the coefficients of noncommutative Bell polynomials (Theorem \ref{partitionBell}). 

Write $N(n; p_1, \dots, p_k)$ for the number of ways to partition a set $[n]$ into parts $P_1, \ldots, P_k$ of sizes $|P_i|= p_i$, such that 
\begin{align}
	\max(P_1) < \max(P_2) < \cdots < \max(P_k).
\end{align} 
Note that if $n > p_1 + \cdots + p_k$ then 
\begin{align}
	N(n; p_1, \ldots, p_k) = {n \choose p^{(k)}} N(p^{(k)}; p_1, \dots, p_k), \quad \text{where  } p^{(k)} = p_1 + \cdots + p_k.
\end{align}

\begin{lemma}~\label{partitionLemma}
The number $N(p^{(k)}; p_1, \cdots, p_k)$ is given by
\begin{align}
\begin{split}~\label{partitionLemma-eq1}
	N(p^{(k)}; p_1, \cdots, p_k) &= \prod_{i=1}^{k-1} {p_1 + \cdots + p_{i+1} - 1 \choose p_1 + \cdots + p_i}\\ 
						&= \prod_{i=1}^{k-1} {p_1 + \cdots + p_{i+1} - 1 \choose p_{i+1} - 1} 
\end{split}
\end{align}
\end{lemma}

\begin{proof}
We have
\begin{align}\label{partitionLemma-eq2}
	N(p^{(k)}; p_1, \dots, p_k) = N(p^{(k)}-1; p_1, \dots, p_{k-1}).
\end{align}
This follows from the next observation. The number $p^{(k)}$ has to be put in the last part $P_k$. There are two options for the second to last number $p^{(k)}-1$: either it goes in the same part $P_k$ or in the part $P_{k-1}$. The number of ways of putting $p^{(k)}-1$ in part $P_{k-1}$ is given by $N(p^{(k)}-1; p_1, \dots, p_{k-1})$ minus the number of ways we can fail to use $p^{(k-1)}-1$ when partitioning the set of size $p^{(k-1)}-1$. Therefore
\begin{align}
	N(p^{(k)}; p_1,\dots, p_k) = N(p^{(k)}-1; p_1, \dots p_k-1) &+ N(p^{(k)}-1; p_1, \dots, p_{k-1})\\ 
		&- N(p^{(k-2)} -2; p_1, \dots, p_{k-1}).
\end{align}
Equation (\ref{partitionLemma-eq2}) then follows because
\begin{align}
	N(p^{(k)}-1; p_1, \dots p_k-1) = N(p^{(k-2)} -2; p_1, \dots, p_{k-1})
\end{align}
by induction. Therefore, since $N(p_1;p_1) = 1$, we get
\begin{align}
& N(p^{(k)}; p_1, \dots, p_k) = N(p^{(k)}-1; p_1, \dots, p_k) \\
\quad &= {p^{(k)}-1 \choose p^{(k-1)}} N(p^{(k-1)}; p_1, \dots, p_{k-1})\\
&= {p_1 + \cdots + p_k -1 \choose p_1 + \cdots + p_{k-1}} N(p^{(k-1)}; p_1, \dots, p_{k-1}) \\
&= {p_1 + \cdots + p_k -1 \choose p_1 + \cdots + p_{k-1}} {p^{(k-1)}-1 \choose p^{(k-2)}} N(p^{(k-2)}; p_1, \dots, p_{k-2})\\ 
&= {p_1 + \cdots + p_{k-1} -1 \choose p_1 + \cdots + p_{k-1}} {p_1 + \cdots + p_{k-2} -1 \choose p_1 + \cdots + p_{k-2}} N(p^{(k-2)}; p_1, \dots, p_{k-2}) \\
&\vdots \\
&= \prod_{i=1}^{k-1} {p_1 + \cdots + p_{i+1} - 1 \choose p_1 + \cdots + p_i} N(p^{(1)}; p_1)
\end{align}
\end{proof}
We will relate this to the noncommutative Bell polynomials via a useful alternative description of the polynomials.

\paragraph{Another formula.}\label{Sect:dynkinBell} For $\omega = d_{j_1} d_{j_2} \cdots d_{j_k}$ we write 
\begin{align}
	{n \choose \omega} = {n \choose {|d_{j_1}|, \ldots, |d_{j_k}|}} = \frac{n!}{j_1! j_2! \cdots j_k!},
\end{align}
and
\begin{align}
	\kappa(\omega) = \kappa(|d_{j_1}|, \ldots, |d_{j_k}|) = \frac{j_1 j_2 \cdots j_k}{j_1(j_1+j_2) \cdots (j_1 + j_2 + \cdots + j_k)}.
\end{align}
Note that the coefficients $\kappa$ form a partition of unity on the symmetric group $S_k$:
\begin{align}
	\sum_{\sigma \in S_k} \kappa(\sigma(\omega)) = 1.
\end{align}

\begin{proposition}[{\cite{lundervold2011hao}}]
\label{explicitBell}
The noncommutative partial Bell polynomials can be written as
\begin{align}
	B_{n,k} = \sum_{|\omega|=n\atop \#(\omega)=k} {n \choose \omega} \kappa(\omega)\omega,
\end{align}
where $\omega = d_{j_1} \cdots d_{j_k}$.
\end{proposition}

\begin{proof}
This follows from the description of the Bell polynomials via the recursion
\begin{align}
	B_{n+1} &= \sum_{k=0}^n {n \choose k} B_{n-k} d_{k+1}\\
		B_0 &= 1.
\end{align}
Let $d_{j_1} \cdots d_{j_k}$ be any monomial in $B_{n,k}$, where $n = j_1 + \cdots + j_k$. The monomial comes from the monomial $d_{j_1} \cdots d_{j_{k-1}}$ in the Bell polynomial $B_{n-j_k}$. By induction its coefficient is
\begin{align}
&{j_1 + \cdots + j_k -1 \choose j_k -1} \cdot \frac{(j_1 + \cdots + j_{k-1})!}{j_1! j_2! \cdots j_{k-1}!} \cdot \frac{j_1 \cdots j_{k-1}}{j_1 \cdots (j_1 + \cdots + j_{k-1})} \\
&\quad =  \frac{n!}{j_1! j_2! \cdots j_k!} \cdot \frac{j_1 \cdots j_{k}}{j_1 \cdots (j_1 + \cdots + j_{k})},
\end{align}
as claimed. 
\end{proof}

This formula is related to the (inverse) \emph{Dynkin idempotent}, see e.g. \cite{lundervold2011hao}. Note that the scaled version of noncommutative Bell polynomials defined in Section \ref{Sect:basicRecursive} can be written as
\begin{align}
Q_{n,k} = \sum_{|\omega|=n \atop \#(\omega) = k} \kappa(\omega)\omega.
\end{align}

\begin{theorem}
\label{partitionBell}
Let $\omega = d_{j_1} \cdots d_{j_k}$. The coefficient ${n \choose \omega} \kappa(\omega)$ of $\omega$ in Proposition \ref{explicitBell} counts the number of partitions of $[n]$, where $n = |\omega|$, into parts $P_1, \dots, P_k$, each of size $|P_i| = j_i$, such that 
\begin{align}
	\max(P_1) < \max(P_2) < \dots < \max(P_k).
\end{align}
\end{theorem}

\begin{proof}
We show that
\begin{align}
	\frac{n!}{j_1! \cdots j_k!} \cdot \frac{j_1 \dots j_k}{j_1 \cdots (j_1 + j_2 + \cdots + j_k)} 
	= \prod_{i=1}^{k-1} {j_1 + \cdots + j_{i+1} - 1 \choose j_1 + \cdots + j_i}
\end{align}
by showing that the left side also satisfies Equation \eqref{partitionLemma-eq2}. This is a straightforward calculation. Note first that for $k=1$ ($n=j_1$) we get $1$. Write $j^{(k)} = n = j_1 + \cdots + j_k$.
\begin{align}
&\frac{j^{(k)}!}{j_1! \cdots j_k!} \cdot \frac{j_1 \cdots j_k}{j_1 \cdots (j_1 + j_2 + \cdots + j_k)} \\
&= \frac{(j^{(k)}-1)!}{j_1! \cdots (j_k-1)!} \cdot \frac{j_1 \cdots j_{k-1}}{j_1 \cdots (j_1 + j_2 + \cdots + j_{k-1})} \\
&\quad = \frac{(j_1 + \cdots + j_{k-1}+1) \cdots (j_1 + \cdots + j_k -1)}{(j_k-1)!}\\ 
& \qquad \qquad \cdot \frac{j^{(k-1)}!}{j_1! \cdots j_{k-1}!}\cdot \frac{j_1 \cdots j_{k-1}}{j_1 \cdots (j_1 + j_2 + \cdots + j_{k-1})}\\
&\quad = {j^{(k-1)}-1 \choose j^{(k-1)}} \cdot \frac{j^{(k-1)}!}{j_1! \cdots j_{k-1}!}\cdot \frac{j_1 \cdots j_{k-1}}{j_1 \cdots (j_1 + j_2 + \dots + j_{k-1})} \\
&\quad = N(j^{(k)}-1; j_1, \ldots, j_{k-1}).
\end{align}

\end{proof}

We arrive at the following formula for the noncommutative Bell polynomials.
\begin{align}
	B_{n,k} = \sum_{\#(\omega) = k} \prod_{i=1}^{k-1} {j_1 + \cdots + j_{i+1} - 1 \choose j_{i+1} - 1} \omega
\end{align}

\begin{remark}[The q-analogs of Bell polynomials]\label{q-Bell}
As an interesting side note, we mention the \emph{q-analogs} of commutative Bell polynomials, constructed by Johnson in \cite{johnson1996sao, johnson1996aqa} based on the work of Gessel (\cite{gessel1982aqa}). The construction is based on q-analogs of integers, defined for any integer $n$ as 
\begin{align}
	[n] := \frac{1-q^n}{1-q} = 1+q+ \cdots + q^{n-1}.
\end{align}
To define the $q$-Bell polynomials we will need the $q$-analogous of factorials:
\begin{align}
	n!_q &= [1][2]\cdots[n] \\
		&= 1 \cdot (1+q) \cdots (1 + q + q^2 + \cdots + q^{n-1}),
\end{align}
and multinomials:
\begin{align}
	{n \brack m_1, m_2, \dots, m_k} &= \frac{n!_q}{m_1!_q m_2!_q \cdots m_k!_q}.
\end{align}

One can define commutative q-Bell polynomials (\cite{johnson1996sao}) as follows. 
\begin{definition} For a word $\omega = d_{p_1}\cdots d_{p_k}$ the $q$-Bell polynomial is
\begin{align}
\label{cqBell}
&B_{n,k,q}(d_1,d_2, \ldots, d_{n-k+1}) \\ &\quad = \sum_{{|\omega| = n} \atop p_i \geq 1} \frac{n!_q}{p_1!_q \cdots p_k!_q} \cdot \frac{p_1\cdots p_k}{[p_1][p_1+p_2] \cdots [p_1 + \cdots + p_k]} \omega\\
& \quad = \sum_{{|\omega| = n} \atop p_i \geq 1} {n \brack p_1, \dots, p_k} \kappa(\omega) \omega.
\end{align}
\end{definition}
\end{remark}

\begin{remark}
Bell polynomials also appear in the study of \emph{noncommutative symmetric functions} \cite{gelfand1995ncs}, where they provide a change of basis. See \cite{gelfand1995ncs}.
\end{remark}

\section{Incidence and Faà di Bruno Hopf algebras} 
\label{sect:FdB}

Commutative Bell polynomials model the composition of formal diffeomorphisms on vector spaces via the Faà di Bruno formula, see equation (\ref{derivativesBell1}). This can be captured algebraically in the Faà di Bruno Hopf algebra, where composition of diffeomorphisms corresponds to convolution (Section \ref{Sect:FdB}). For diffeomorphisms on more general manifolds the noncommutative Bell polynomials play an analogous role, and give rise to the Dynkin Faà di Bruno Hopf algebra (Section \ref{Sect:DFdB}), first studied in \cite{lundervold2011hao}.

\subsection{Faà di Bruno formulas}\label{Sect:FdBform}

Recall that the $n$-th derivative of a composition $f \circ g$ can be written using the well-known Faà di Bruno formula \cite{faadibruno1855ssd, faadibruno1857nsu}\footnote{Faà di Bruno was not the first to express derivatives of $f\circ g$ in this way, but his result is the most well-known. Earlier results can be found in \cite{ta1850sld, arbogast1800dcd}. See \cite{johnson2002tch} for a historical account of the Faà di Bruno formula.}:
\begin{align}
\frac{d^n}{dx^n} f(g(x)) = \sum_{j_1 + \cdots + j_n = k \atop j_1 + \cdots + nj_n = n} \frac{n!}{j_1! \cdots j_n!} f^{(k)}(g(x)) \left(\frac{g'(x)}{1!}\right)^{j_1} \cdots \left(\frac{g^{(n)}(x)}{n!}\right)^{j_n},
\end{align}
where all the necessary derivatives are assumed to exist. This is highly reminiscent of Bell polynomials. Indeed, we can write
\begin{align}
\label{derivativesBell1}
	\frac{d^n}{dx^n} f(g(x)) = \sum_{k=0}^n f^{(k)}(g(x)) B_{n,k}(g'(x), g''(x), \dots, g^{(n-k+1)}(x)),
\end{align}
where $B_{n,k}$ are the commutative Bell polynomials, and $B_{n,0} = 0$ for $n \neq 0$. This is known as \emph{Riordan's formula} (\cite{riordan1946doc}).

\paragraph{Examples:}
\begin{align}
	\frac{d}{dx} f(g(x)) 		&= f'(g(x)) B_{1,1}(g'(x)) = f'(g(x)) g'(x) \\
	\frac{d^2}{dx^2} f(g(x)) 	&= f'(g(x))B_{2,1}(g'(x), g''(x)) + f''(g(x)) B_{2,2}(g'(x),g''(x)) \\ 
						& = f'(g(x))g''(x) + f''(g(x)) (g'(x))^2.
\end{align}
We shall see how this can be formulated in terms of composition of diffeomorphisms on vector spaces, where the coefficients of the composition $f \circ g$ of two diffeomorphisms are given in terms of the coefficients of $f$ and $g$. The noncommutative Bell polynomials will be shown to correspond to composition of diffeomorphisms on manifolds. 

\begin{remark}
From the description of Bell polynomials as determinants in Section \ref{Sect:quasideterminants} we obtain
\begin{align}
\frac{d^n}{dx^n} f(g(x)) = \left|\begin{array}{cccccc} 
{0 \choose 0} g'f & {1 \choose 1} g''f & {2 \choose 2} g'''f & \cdots & {n-2 \choose n-2} g^{(n-1)}f & {n-1 \choose n-1} g^{(n)}f \\ \\
-1 & {1 \choose 0} g'f & {2 \choose 1} g''f & \cdots &  {n-2 \choose n-3} g^{(n-2)}f & {n-1 \choose n-2} g^{(n-1)}f \\ \\
0 & -1 & {2 \choose 0} g'f & \cdots & {n-2 \choose n-4} g^{(n-3)}f & {n-1 \choose n-3} g^{(n-2)}f\\ \\
\vdots & \vdots & \vdots & & \vdots & \vdots\\ \\ 
0 & 0 & 0 & \cdots & {n-2 \choose 0} g'f & {n-1 \choose 1} g''f \\ \\
0 & 0 & 0 & \cdots & -1 & {n-1 \choose 0} g'f
\end{array}\right|
\end{align}
This formula was first discovered by Faà di Bruno (\cite{faadibruno1855ssd}).
\end{remark}

\begin{remark}[A q-analog]
As mentioned in Remark \ref{q-Bell}, \cite{johnson1996aqa} develops q-analogs of Bell polynomials, which can be used in a q-analog of the Faà di Bruno formula. The formula can be written as a sum indexed over partitions:
\begin{align}\label{qFdB1}
\mathbf{D}^n_q g[f(x)] &= \sum_{P \in \mathcal{P}(n)} q^{w(P)} g^{(k)}[f(x)] f^{(p_1)}(x) f^{(p_2)}(q^{p^{(2)}}x) \cdots f^{(p_k)}(q^{p^{(k)}}x),
\end{align}
where $p_i:=|P_i|$, and $p^{(j)}:=p_1 + \cdots + p_{j-1}$. The \emph{weight} $w(P)$ of the partition $P$ is calculated as follows: In a partition $P$ of $[n]$ into blocks $P_1, P_2, \cdots P_k$, iteratively cross out the block $P_{\max}^1$ whose maximum is largest, relabel the numbers in the remaining partitions in an order preserving manner from $1$ to $|P_{\max}^1|$. The number of relabelings is called $r_1$. Continue in this manner until there's only one block remaining. The weight of $P$ is the sum of the $r_i$. For example, the weight of the partition $P$ with blocks
$\{1,2,7\}, \{3,6\}, \{4,5\}, \{8,9,13,14\}, \{10,12\}, \{11\}$ is $9$.

The following lemma allows us to rewrite Johnson's q-Faà di Bruno formula in a more familiar form.

\begin{lemma}[{\cite{johnson1996aqa}}]
Let $P$ be a partition of $[n]$ into the blocks $\{P_1, \dots, P_k\}$, listed in increasing order of their maximal elements, with $|P_i| = p_i$. Then
\begin{align}
&\sum_{P \in \mathcal{P}_n} q^{w(P)} \\ & \quad = {p_1 + p_2 + \cdots + p_k - 1 \brack p_k-1}_q {p_1 + p_2 + \cdots + p_{k-1} - 1 \brack p_{k-1}-1}_q \cdots {p_1 + p_2 \brack p_2-1}_q
\end{align}
\end{lemma}

Using the q-Bell polynomials defined in Remark \ref{q-Bell}, the q-analogue of the Faà di Bruno formula \eqref{qFdB1} can be written as
\begin{align}\label{qFdB2}
\mathbf{D}^n_q g[f(x)] = \sum_{ p_1 + \cdots + p_k =n \atop p_i \geq 1} g^{(k)}[f(x)] B_{n,k,q}(f_{p_1,0}, f_{p_2,p_1}, f_{p_3,p_1+p_2}, \dots),
\end{align}
where $f_{i,j} := f^{(i)}(q^j x)$.
\end{remark}

\subsubsection{Composition of formal diffeomorphisms on vector spaces.} 

Composition of smooth and invertible functions on the real line $\mathbb{R}$ forms a group, called the group of diffeomorphisms on $\mathbb{R}$. Following \cite{figueroa2005fdb, frabetti2014fio}, we consider the group write $G = \Diff(\mathbb{R})$ of \emph{formal} diffeomorphisms leaving the origin fixed:
\begin{align}
	G = \left\{f(t) = \sum_{n=1}^{\infty} \frac{f_n}{n!} t^n, \quad f_0 = 0, \,f_1 > 0 \right\}.
\end{align}
Let $h = f\circ g$ be a composition:
\begin{align}
	h(t) = \sum_{k=1}^{\infty} \frac{f_k}{k!} \left(\sum_{l=1}^{\infty} \frac{g_l}{l!} t^l\right)^k.
\end{align}
The Cauchy product formula gives
\begin{align}
	h(t) = \sum_{k=1}^{n} \frac{f_k}{k!}\sum_{ l_1 + \dots + l_k = n \atop l_i \geq 1} 
	\frac{n! g_{l_1}\cdots g_{l_k}}{l_1! \cdots l_k!}.
\end{align}
In other words,
\begin{align}
	h_n = \sum_{k=1}^n f_k \sum_{\lambda}\frac{n!}{\lambda_1! \cdots \lambda_n!} 
	\frac{g_1^{\lambda_1} \cdots g_n^{\lambda_n}}{(1!)^{\lambda_1} (2!)^{\lambda_2} \cdots (n!)^{\lambda_n}},
\end{align}
or
\begin{align}
	h^{(n)}(t) = \sum_{k=1}^n \sum_{\lambda} \frac{n!}{\lambda_1! \cdots \lambda_n!} f^{(k)}(g(t)) \left(\frac{g^{(1)}(t)}{1!}\right)^{\lambda_1} \cdots \left(\frac{g^{(n)}(t)}{n!}\right)^{\lambda_n},
\end{align}
where 
\begin{align}
	\lambda_1 + 2\lambda_2 + \dots + n\lambda_n = n, \quad \text{with  } \lambda_1 + \dots + \lambda_n = k.
\end{align}
This can also be written as
\begin{align}
	h_n = \sum_{k=1}^n f_k B_{n,k}(g_1, \dots, g_{n+1-k}),
\end{align}
or
\begin{align}
	h^{(n)}(t) = \sum_{k=1}^n f^{(k)}(g(t)) B_{n,k}(g^{(1)}(t), \dots, g^{(n-k+1)}(t)),
\end{align}
which is the same as Formula \eqref{derivativesBell1}. For example,
\begin{align}
	h_1 &= f_1 B_{1,1}(g_1) = f_1 g_1 \\
	h_2 &= f_1 B_{2,1}(g_1,g_2) + f_2 B_{2,2}(g_1,g_2) = f_1 g_2 + f_2 (g_1)^2.
\end{align}
See \cite{figueroa2005fdb} or \cite{frabetti2014fio} for more details.

\subsubsection{Composition of diffeomorphisms on manifolds}
\label{Sect:FdBManifolds}

We shall see how noncommutative Bell polynomials model the composition of time-dependent flows on manifolds. We merely describe the main constructions. For more details, consult \cite{lundervold2011hao, munthe-kaas1995lbt}. For background material about the relevant constructions from differential geometry, see e.g. \cite{marsden2007mta, sharpe1997dg}. 

The derivation operation we consider is the so-called \emph{Lie derivative}:
\begin{definition}
The \emph{Lie derivative} of a function $\psi: M \rightarrow \mathbb{R}$ along a vector field $F \in \mathcal{X}(M)$ is
\begin{align}
	F[\psi](x) := \mathbf{d} \psi(x) \cdot F(x),
\end{align}
where $\mathbf{d} \psi: M \rightarrow T^*M$ is the differential of $\psi$.
\end{definition}

Note that if $M$ is finite dimensional then (in local coordinates)
\begin{align}
	F(x) = \sum_{i=1}^n F^i(x) \frac{\partial}{\partial x^i},
\end{align}
and
\begin{align}
	(\mathbf{d} \psi)_i = \frac{\partial \psi}{\partial x^i}, \quad F[\psi] = \sum_{i=1}^n F^i \frac{\partial \psi}{\partial x_i}.
\end{align} 
Here $\{\frac{\partial}{\partial x^i}\}$ spans the space $\mathcal{X}(M)$ of vector fields in the local coordinates $\{x_1, \dots, x_n\}$. If $M$ is \emph{parallelizable} then this is a global basis (e.g.~for any Lie group). Note further that if $M = \mathbb{R}$ then
\begin{align}
	F[\psi] = \psi'(x) F(x).
\end{align}

We want to define the Lie derivative $F[G]$ of a vector field $G$ along another vector field $F$. Let $\Phi_{t,s}: M \rightarrow M$ be the \emph{flow} of $F_t$, i.e. the diffeomorphisms $\Phi_{t,s}$ such that $t \mapsto \Phi_{t,s}$ is the integral curve of $F$ starting at $m$ at time $t=s$:  
\begin{align}
	\frac{d}{dt} \Phi_{t,s} = F_t(\Phi_{t,s}), \quad \Phi_{s,s}(m) = m.
\end{align}
Write $\Phi_{t,F}$ for the flow of the vector field $F$ starting at time $t=0$. Let $\psi: M \rightarrow \mathcal{V}$ be a section of a trivial bundle over $M$. We form the pullback of $\psi$ along the flow $\Phi_{t,F}$:
\begin{align}
	\Phi_{t,F}^* \psi := \psi \circ \Phi_{t,F}.
\end{align}
We are interested in computing its derivatives.

\begin{definition}
Let $F,G \in \mathcal{X}(M)$ be two vector fields, and write $\Phi_{t,F}$ for the flow of $F$. The \emph{Lie derivative} of $G$ with respect to $F$ is defined by
\begin{align}
	F[G] := \frac{d}{dt}\bigg|_{t=0} \Phi_{t,F}^* G.
\end{align}
\end{definition}
Note that the Lie derivative is a derivation: if $\phi: M \rightarrow \mathbb{R}$ and $G \in \mathcal{X}(M)$ then 
\begin{align}
	F[\psi G] = F[\psi]G + \psi(F[G]).
\end{align}
Composition of Lie derivatives gives a (associative, noncommutative) product on the space $\mathcal{X}(M)$ of vector fields on $M$. 

Vector fields are invariant under their own flow, $\Phi_{t,F}^* F = F$, so  
\begin{align}
	F[F] = \frac{d}{dt}\bigg|_{t=0} \Phi_{t,F}^*F = \frac{d}{dt}\bigg|_{t=0} F.
\end{align}
The basic derivative formula is

\begin{align}\label{basicPullback}
	\frac{d}{dt} \Phi_{t,F}^* \psi = \Phi_{t,F}^*(F[\psi]),
\end{align}
which follows from a simple application of the chain rule (see \cite[Theorem 4.2.31]{marsden2007mta}). In particular,
\begin{align}
	\frac{d}{dt}\bigg|_{t=0} \Phi_{t,F}^* \psi = F[\psi].
\end{align}
By iterating Formula \eqref{basicPullback} we get
\begin{align}
	\frac{d^n}{dt^n}\bigg|_{t=0} \Phi_{t,F}^* = F[F[\cdots F[\psi]\cdots]] = F^n[\psi],
\end{align}
and the Taylor expansion of the pullback can be written as
\begin{align}
	\Phi_{t,F}^*\psi = \psi + tF[\psi] + \frac{t^2}{2!} F[F[\psi]] + \cdots.
\end{align}
This can be formulated in terms of the noncommutative Bell polynomials.

\begin{theorem}[{\cite{lundervold2011hao}}]\label{pullbackAndBell}
We have
\begin{align}
	\frac{d^n}{dt^n} \Phi_{t,F}^* \psi = B_n(F)[\psi],
\end{align}
where $B_n(F)$ is the image of the noncommutative Bell polynomials $B_n(d_1, \ldots, d_n)$ under the map $d_i \mapsto F^{(i-1)}$. In particular
\begin{align}
	\frac{d^n}{dt^n}\bigg|_{t=0} \Phi_{t,F^t}^* \psi = B_n(F_1, \ldots, F_n)[\psi],
\end{align}
where $F_{n+1} = F^{(n)}(0)$.
\end{theorem}

For example,  
\begin{align}
\frac{d}{dt}\bigg|_{t=0} \Phi_{t,F}^* \psi &= B_1(F_1)[\psi] = F_1[\psi]\\
\frac{d^2}{dt^2}\bigg|_{t=0} \Phi_{t,F}^* \psi &= B_2(F_1, F_2)[\psi] = F_1^2[\psi] + F_2[\psi]\\
\frac{d^3}{dt^3}\bigg|_{t=0} \Phi_{t,F}^* \psi &= B_3(F_1, F_2, F_3)[\psi] = F_1^{3}[\psi] + (F_2F_1)[\psi] + 2(F_1F_2)[\psi] + F_3[\psi]
\end{align}

\begin{remark}
Let $M = \mathbb{R}$. Then $\psi: \mathbb{R} \rightarrow \mathbb{R}$, $F:\mathbb{R} \rightarrow \mathbb{R}$, and the Lie derivative is $F[\psi] = \psi'(x)F(x)$, so Formula \eqref{derivativesBell1} and the formula in Theorem \eqref{pullbackAndBell} agree.
\end{remark}

\subsection{Faà di Bruno Hopf and bialgebras.}
\label{Sect:incidenceHopf}

This section contains descriptions of commutative and noncommutative Faà di Bruno Hopf algebras, both constructed as incidence Hopf algebras and directly from the Bell polynomials. We start with a short presentation of incidence Hopf algebras.

\subsubsection{Incidence Hopf algebras}

Incidence Hopf algebras have been defined and studied intensively by W.~Schmitt \cite{schmitt1994iha}, starting from the notions of incidence algebra \cite{rota1964otf} and incidence bi- and coalgebras \cite{joni1979cab}. The framework incorporates various combinatorial Hopf algebras, such as symmetric functions, the Butcher--Connes--Kreimer Hopf algebra of rooted forests, Hopf algebras of finite posets, and various Fa\`a di Bruno Hopf algebras.

A \emph{poset} is a partially ordered set $P$ with an order relation, which we denote the by $\le$. For any $x,y \in P$, the \emph{interval} $[x,y]$ is the subset of $P$ formed by the elements $z$ such that $x\le z\le y$. Let $\P$ be a family of finite posets $P$ such that there exists a unique minimal element $0_P$ and a unique maximal element $1_P$ in $P$ (hence the poset $P$ coincides with the interval $P=[0_P,1_P]$). The family is called \emph{interval closed} if for any poset $P \in\P$ and for any $x \le y \in P$, the interval $[x,y]$ is an element of $\P$.

From the family $\P$ one can construct a coalgebra by considering equivalence classes of elements under an \emph{order-compatible} relation $\sim$ on $\P$. That is, $P \sim Q \in \P$ if there exists a bijection $\varphi: P \to Q$ such that:
\begin{align}
	[0_P,x] \sim [0_Q,\varphi(x)] \hbox { and } [x,1_P]\sim[\varphi(x),1_Q]
\end{align}
for any $x \in P$. An obvious example of order-compatible equivalence relation is poset isomorphism, but it is useful to consider more general situations.

Let $\overline{\P}$ be the quotient $\P/\sim$, where $\sim$ is an order-compatible relation. The equivalence class of any poset $P\in\P$ is denoted by $\overline P$ (notation borrowed from \cite{ehrenborg1996opa}). The \emph{incidence coalgebra} of the family of posets $\P$ together with the equivalence relation $\sim$ is the $\K$-vector space freely generated by $\overline{\P}$, with coproduct given by 
\begin{equation*}
	\Delta(\overline P)=\sum_{x\in P}\overline{[0_P,x]}\otimes \overline{[x,1_P]}, 
\end{equation*}
and counit given by $\varepsilon(\overline{\{*\}})=1$ and $\varepsilon(\overline P)=0$ if $P$ contains two elements or more.

Given two posets $P$ and $Q$, the \emph{direct product} $P \times Q$ is the set-theoretic cartesian product of the two posets, with partial order given by $(p,q) \leq (p',q')$ if and only if $p \leq p'$ and $q \leq q'$. A family of finite posets $\P$ is called \emph{hereditary} if the product $P \times Q$ belongs to $\P$ whenever $P,Q \in \P$. An order-compatible equivalence relation $\sim$ on $\P$ is \emph{reduced} if $P \times Q \sim Q \times P \sim P$ whenever $Q$ is a one-element set.

The quotient $\overline{\P}$ is then a semigroup generated by the set $\overline{\P_0}$ of classes of \emph{indecomposable posets}, i.e. posets $R\in\P$ such that for any $P,Q\in\P$ of cardinality greater than one, $P \times Q$ is not isomorphic to $R$. The unit element $\un$ is the class of any poset with only one element.

\begin{theorem}[{\cite[Theorem 4.1]{schmitt1994iha}}]
If $\P$ is a hereditary family of finite posets and $\sim$ a reduced order-compatible semigroup relation, then the associated incidence coalgebra $H(\P)$ is a Hopf algebra.
\end{theorem}

Note that when $\sim$ is the equivalence relation given by poset isomorphism, the obvious equivalence $P \times Q \sim Q \times P$ for any $P,Q\in \P$ shows that the incidence Hopf algebra is commutative in this case. This is the \emph{standard reduced incidence Hopf algebra} associated with the hereditary family of posets $\P$.

\subsubsection{Antipodes, uniform families and quasideterminants}
\label{determinantAntipode}

Various formulas for antipodes for incidence Hopf algebras subject to some restrictions have been developed, e.g. in \cite{haiman1989iaa, schmitt1987aai, figueroa2005cha, einziger2010iha}. One particularly useful general formula was given in \cite[Theorem 4.1]{schmitt1994iha}:
\begin{align}
	S(\overline{P}) = \sum_{k \geq 0} \sum_{\underset{\underset{x_k = 1_P}{x_0=0_P}}{x_0 < \cdots < x_k}} 
	(-1)^k \prod_{i=1}^k [x_{i-1}, x_i],
\end{align}
for $\overline{P} \in \overline{\P}$. 

For a particular class of posets the antipode can be written as a determinant. More precisely, in \cite[Section 8]{schmitt1994iha} Schmitt defined so-called (commutative) \emph{uniform families} of hereditary posets $\mathcal{P}$ and gave a determinantal formula for the antipodes in the associated commutative incidence Hopf algebras. We extend his definition of uniform families to the noncommutative case, and show that the antipode formula extends to the quasideterminant. 

A uniform family will consist of graded posets. A poset $P$ is called \emph{graded} if all the \emph{chains}, i.e. sets of elements  $0_P = x_1 < x_2 < \cdots < x_n  = 1_P$ in $P$ satisfying 
\begin{align}
	x_i \leq y \leq x_{i+1} \quad \Longrightarrow \quad y=x_i \text{ or } y = x_{i+1}, \quad \forall 1\leq i \leq n-1,
\end{align}
are of the same length $\rank(P)=n$. This common length is called the \emph{rank} of $P$. In a hereditary family $\P$ of graded posets the rank function is well-defined on the quotient $\overline{\P}$, because the equivalence relation is order-compatible. 

\begin{definition}
\label{uniform}
A (commutative or noncommutative) \emph{uniform family} is a hereditary family $\mathcal{P}$ of graded posets together with a reduced order-compatible relation $\sim$ such that
\begin{itemize}
	\item[(1)] If $\overline{P} \in \overline{\mathcal{P}}_0,$ $y \in P$ and $y < 1_P$, then $[y,1_P] \in \overline{\mathcal{P}}_0$.
	\item[(2)] For all $n \geq 1$ there exists exactly one type in $\overline{\mathcal{P}}_0$ having rank $n$.
\end{itemize}
\end{definition}

Let $x_n$ be the unique indecomposable type of rank $n$, $n\geq 1$, and $x_0 = 1$. Then $H(\mathcal{P})$ is isomorphic as a graded algebra to the free associative algebra $\mathbb{K}\langle x_1, x_2, \dots \rangle$, where $\deg(x_n) = n$. Following \cite{schmitt1994iha}, we define the \emph{rank polynomial} $W_{n,k} = W_{n,k}(x_1, x_2, \dots)$ in $H(\mathcal{P})$ by $W_{0,0} = 1$, and for $n \geq 1$, by choosing $[x,y]$ of rank $n$ in $\overline{\mathcal{P}}_0$, and setting
\begin{align}
	W_{n,k} = \sum_{\underset{r[z,y]=k}{z \in [x,y]}} [x,z].
\end{align}
Note that $W_{n,n} = 1$ and $W_{n,0} = x_n$ for $n \geq 0$, and $W_{n,k} = 0$ for $n<k$, and that the coproduct in $H(\mathcal{P})$ can be written as
\begin{align}
	\Delta(x_n) = \sum_{k\geq 0} W_{n,k}\otimes x_k.
\end{align}

Write $M_n$ for the matrix whose $i$th row and $j$th column is $W_{n-i+1, n-j}$, and put $M_0 = I$, the $1\times 1$ identity matrix. For example,
\begin{align}
M_4 = \left(\begin{array}{cccc}
W_{4,3} & W_{4,2} & W_{4,1} & W_{4,0} \\
1 & W_{3,2} & W_{3,1} & W_{3,0}\\
0 & 1 & W_{2,1} & W_{2,0} \\
0 & 0& 1& W_{1,0}
\end{array}\right)
\end{align}

\begin{theorem}
\label{incidenceAntipode}
If $\mathcal{P}$ is a noncommutative uniform family, then the antipode $S$ of $H(\mathcal{P})$ can be written as 
\begin{align}
	S(x_n) = \big|M_n|_{1n},
\end{align}
where $|\cdot|_{1n}$ is the quasideterminant computed at the top right element.
\end{theorem}

\begin{proof}
The proof mimics the one for commutative incidence Hopf algebras in \cite{schmitt1994iha}. Define the algebra map $S':H(\mathcal{P}) \rightarrow H(\mathcal{P})$ by 
\begin{align}
	S'(x_n) = |M_n|_{1n}.
\end{align}
We want to show that 
\begin{align}
	\mu \circ (id \otimes S') \circ \Delta(x_n) = 0,
\end{align}
for all $n\geq 1$. In other words, that
\begin{align}
	\sum_{k= 0}^n W_{n,k}S'(x_k) = 0.
\end{align}

By Proposition \ref{quasidetExpansion}, with $a_{ij} = -W_{n-i+1, n-j}$, we have 
\begin{align}
	|M_n|_{1n} = \sum_{k=0}^{n-1} -W_{n,k} |M_k|_{1k}.
\end{align} 
We get
\begin{align}
	S'(x_n) = \sum_{k=0}^{n-1} -W_{n,k} S'(x_k),
\end{align}
so 
\begin{align}	
	0 = \sum_{k=0}^n W_{n,k}S'(x_k).
\end{align}
By uniqueness of the antipode, the result follows.
\end{proof}

Note that if the uniform family is commutative we obtain the determinantal formula of Schmitt (by Remark \ref{commQuasiPoly}). 
\begin{align}
	S(x_n) = (-1)^n |M_n|.
\end{align}

\begin{remark}
One can recover the so-called Möbius function from the zeta function in an incidence Hopf algebra by composing with the antipode:
\begin{align}
	\mu_P = \zeta_P \circ S_P,
\end{align}
viewed as elements in the incidence algebra associated to $P$. The quasideterminantal formulation above then gives a new description of the Möbius function in noncommutative incidence Hopf algebras. This procedure will be exemplified for variants of the Faà di Bruno Hopf algebras in Section \ref{Sect:mobius}.

\end{remark}


\subsubsection{The commutative Faà di Bruno Hopf algebra}
\label{Sect:FdB}

The commutative Faà di Bruno Hopf algebra has been described many times in the literature, see e.g. \cite{joni1979cab, figueroa2005fdb, frabetti2014fio}. We will give a quick refresher, describing it both directly as a Hopf algebra on a polynomial ring and as an incidence Hopf algebra.

The Faà di Bruno Hopf algebra is the graded polynomial ring $\mathbb{K}[x_1, x_2, \dots]$, where $\deg(x_n) = n$. The counit is $\epsilon(x_n) = \delta_{n,0}$, where $x_0 = 1$ and the coproduct is given by
\begin{align}\label{coprodFdB}
	\Delta(x_n) &= \sum_{k=0}^n \left(\sum_{\underset{\scriptstyle k_1 + 2k_2 + \cdots + nk_n = n-k}{k_0 + k_1 + \cdots + k_n = k+1}} 
	\frac{(k+1)!}{k_0! k_1! \cdots k_n!} x_1^{k_1} \cdots x_n^{k_n}\right)\otimes x_k \\
			 &= \sum_{k=0}^n \frac{(k+1)!}{(n+1)!}B_{n+1,k+1}(x_0,2!x_1, 3!x_2, \ldots) \otimes x_k,
\end{align}
where $B_{n+1,k+1}$ are the commutative partial Bell polynomials, and $x_0 = 1$. The coproduct is extended multiplicatively. 

\begin{remark} Note that a very simple way to encode the coproduct on the generators $x_n$ results by considering the element $x:= 1+ \sum_{n>0} t^n x_n$. In fact, one can show that 
\begin{align}
	\Delta(x) = \sum_{n \ge 0} x^{n+1} \otimes x_n.	
\end{align}
\end{remark}

By turning to a new set of generators $X_j:=(j+1)!x_j$ the formula in (\refeq{coprodFdB}) simplifies a bit:
\begin{align}\label{coprod-fdb}
	\Delta(X_n)=\sum_{k=0}^n B_{n+1,k+1}(X_0,X_1, X_2, \ldots) \otimes X_k.
\end{align}
We obtain:
\begin{align}
\begin{split}
\Delta(X_0) &= X_0 \otimes X_0 \\
\Delta(X_1) &= X_1 \otimes X_0 + X_0 \otimes X_1\\
\Delta(X_2) &= X_2 \otimes X_0 + X_0 \otimes X_2 + 3X_1 \otimes X_1 \\
\Delta(X_3) &= X_3 \otimes X_0 + X_0 \otimes X_3 + (3X_1^2 + 4X_2) \otimes X_1 + 6X_1 \otimes X_2\\
\Delta(X_4) &= X_4\otimes X_0+X_0\otimes X_4+(10 X_1X_2+5X_3)\otimes X_1+(10 X_2+15 X_1^2)\otimes X_2 \\ &\quad + 10X_1\otimes X_3.
\end{split}
\end{align}
Note that $X_0$ is considered an idempotent. The bialgebra is graded and connected, and therefore automatically a Hopf algebra, denoted by $\H_{FdB}$. Using the Sweedler notation, $\Delta(X_n) - X_n \otimes X_0 - X_0 \otimes X_n:=\sum_{(X_n)}X_n' \otimes X_n''$, the antipode is given recursively as (see e.g. \cite{manchon2006haf}):
\begin{align}
S(X_n) &= -X_n - \sum_{(X_n)} X_n' S(X_n'')\\ 
&= -X_n - \sum_{(X_n)} S(X_n')X_n''.
\end{align}
For example:
\begin{align}
S(X_1) &= -X_1 \\
S(X_2) &= -X_2 + 3X_1^2 \\
S(X_3) &= -X_3 + 10X_1X_2 - 15X_1^3\\
S(X_4)&=-X_4+15X_1X_3+10X_2^2-105X_1^2X_2+105X_1^4.
\end{align}

\paragraph{The Faà di Bruno Hopf algebra as an incidence Hopf algebra.}

Let $\mathcal {SP}$ be the family of posets isomorphic to the set $\mathcal {SP}(A)$ of all partitions of some nonempty finite set $A$. The partial order on set partitions is given by refinement: $S \leq T$ if and only if all the blocks of $S$ are contained in blocks of $T$. We denote by $0_A$ or $0$ the partition by singletons, and by $1_A$ or $1$ the partition with only one block. Let $\mathcal Q$ be the family of posets isomorphic to the cartesian product of a finite number of elements in $\mathcal {SP}$. If $S$ and $T$ are two partitions of a finite set $A$ with $S\le T$, the partition $S$ restricts to a partition of any block of $T$. Denoting by $W/S$ the set of those blocks of $S$ which are included in some block $W$ of $T$, any partition $U$ such that $S\le U\le T$ yields a partition of the set $W/S$ for any block $W$ of $T$. This in turn yields the following poset isomorphism:
\begin{equation}
\label{eq:partitions}
	[S,T]\sim\prod_{W\in A/T} \mathcal {SP}(W/S).
\end{equation}
This shows that $\mathcal Q$ is interval closed (and hereditary by definition). Reciprocally, any element of $\mathcal Q$, isomorphic to the cartesian product of, say, $k$ elements of $\mathcal SP$, is obviously isomorphic to an interval $[0, P]$ where $P$ is a (well-chosen) partition of a finite set into $k$ blocks.

\begin{proposition}[{\cite[Example 14.1]{schmitt1994iha}}]
The standard reduced incidence Hopf algebra $H(\mathcal Q)$ is isomorphic to the Fa\`a di Bruno Hopf algebra.
\end{proposition}

\begin{proof}
Denote by $X_n$ the isomorphism class of $\mathcal{SP}(\{1,\ldots, n+1\})$. Note that $X_0=1$ is the unit.
In view of \eqref{eq:partitions}, we have:
\begin{eqnarray}\label{fdbsp}
	\Delta(X_n)	&=& \sum_{S\in\mathcal{SP}(\{1,\ldots, n+1\})} \overline{[0,S]}\otimes\overline{[S,1]}\notag\\
			 	&=& \sum_{S\in\mathcal{SP}(\{1,\ldots, n+1\})} \left(\prod_{W\in \{1,\ldots , n+1\}/S}\overline{\mathcal{SP}(W)}\right) 
				\otimes \overline{\mathcal{SP}(\{1,\ldots, n+1\}/S)}.
\end{eqnarray}
The coefficient in front of $X_1^{k_1}\cdots X_n^{k_n}\otimes X_k$ in \eqref{fdbsp} above is equal to the number of partitions of $\{1,\ldots ,n+1\}$ with $k_j$ blocks of size $j+1$ (for $j=1$ to $n$), $k+1$ blocks altogether, and $k_0=k+1-k_1-\cdots -k_n$ blocks of size $1$, which in turn gives back \eqref{coprod-fdb}:
{\small
\begin{align}
\begin{split}
	\Delta (X_n)	&=\sum_{k=0}^n B_{n+1,k+1}(X_0,X_1,\ldots ,X_n)\otimes X_k \\ 
				&=\ \sum_{k=0}^n 
	\left(\sum_{{\scriptstyle k_0+k_1+\cdots +k_n=k+1, \atop \scriptstyle k_1+2k_2+\cdots +nk_n=n-k}} 
	\dfrac{1}{k_0! k_1!\cdots k_n!}\ \left(\frac{X_1}{2!}\right)^{k_1}\cdots \left(\frac{X_n}{(n+1)!}\right)^{k_n}\right) 
	\otimes X_k.
\label{doubilet}
\end{split}
\end{align}
}This is the formula for the coproduct in the Faà di Bruno Hopf algebra, modulo the base change $x_j:=\frac{d_j}{(j+1)!}$.
\end{proof}

It follows that the partial commutative Bell polynomials are the rank polynomials (defined in Section \ref{determinantAntipode}) of the commutative Faà di Bruno incidence Hopf algebra:
\begin{align}
	W_{n,k} = B_{n+1,k+1}(X_0,X_1,X_2,\dots), \qquad X_0 = 1,
\end{align}
and Theorem \ref{incidenceAntipode} therefore gives the following description of the antipode.
\begin{theorem}[{\cite[Example 14.1]{schmitt1994iha}}] The antipode in the commutative Faà di Bruno Hopf algebra can be written as
\begin{align} 
	S(x_n) &= (-1)^n \det\big(B_{n-i+2, n-j+1}(1,x_2,x_3, \dots)\big)_{1\leq i,j \leq n},
\end{align}
where $B_{n,k}$ are the commutative partial Bell polynomials.
\end{theorem}
For example, 
\begin{align}
S(x_4) 	&= \left|\begin{array}{ccc} 6x_2 & 4x_3 + 3x_2^2 & x_4 \\ 1 & 3x_2 & x_3 \\ 0 & 1 & 2x_2 \end{array}\right| \\
	&= -x_4 + 10x_2x_3 - 15x_2^3.
\end{align}

\begin{remark} One may ask whether there is a q-version of the commutative Faà di Bruno Hopf algebra based on the q-Bell polynomials of \cite{johnson1996sao} (see Remark \ref{q-Bell}). Unfortunately, such a construction does not seem to be possible. From an incidence Hopf algebra point of view the problem arises because the \emph{weight} of the partitions in \cite{johnson1996sao} is not compatible with the partial order by refinement. Furthermore, q-composition is not associative, making a possible corresponding Hopf algebra quite unwieldy.
\end{remark}

\subsubsection{The noncommutative Dynkin-Faà di Bruno Hopf algebra}
\label{Sect:DFdB}

Consider the alphabet $\mathcal{A} = \{X_n\}_{n\geq 1}$. The Dynkin-Faà di Bruno Hopf algebra $\H_{DFdB}$ is the free associative algebra $\H_{DFdB} = \mathbb{K}\langle X_1, X_2, \dots \rangle$ with unit $X_0$, equipped with the coproduct
\begin{align}
	\Delta(X_n) = \sum_{k=0}^n B_{n+1,k+1} (X_0,X_1,\ldots, X_n)\otimes X_k,
\end{align}
where $B_{n+1,k+1}$ are the partial noncommutative Bell polynomials, extended multiplicatively. The counit is $\epsilon(X_n) = \delta_{n,0}$. It is graded by $|X_n| = n$, and is also connected. We have:
\begin{align}
\begin{split}
	\Delta(X_0) &= X_0 \otimes X_0 \\
	\Delta(X_1) &= X_1 \otimes X_0 + X_0 \otimes X_1\\
	\Delta(X_2) &= X_2 \otimes X_0 + X_0 \otimes X_2 + 3X_1 \otimes X_1 \\
	\Delta(X_3) &= X_3 \otimes X_0 + X_0 \otimes X_3 + (3X_1^2 + 4X_2) \otimes X_1 + 6X_1 \otimes X_2\\
	\Delta(X_4) &=X_4\otimes X_0+X_0\otimes X_4+(6 X_1X_2+4X_2X_1+5X_3)\otimes X_1  \\ 
			  &\quad+(10 X_2+15 X_1^2)\otimes X_2 +10X_1\otimes X_3.
\end{split}
\end{align}
The first disparity between the commutative and noncommutative case appears in $\Delta(X_4)$, where the term $10X_1X_2$ splits into $6X_1X_2+4X_2X_1$. Being a graded and connected bialgebra we immediately have a recursive formula for the antipode in $\H_{DFdB}$:
\begin{align}
S(X_n) &= -X_n - \sum_{(X_n)} X_n' S(X_n'')\\ 
&= -X_n - \sum_{(X_n)} S(X_n')X_n''. 
\end{align}
We get:
\begin{align}
\begin{split}
S(X_1) &= -X_1 \\
S(X_2) &= -X_2 + 3X_1^2 \\
S(X_3) &= -X_3 + 6X_1X_2 +4X_2X_1- 15X_1^3\\
S(X_4)&=-X_4+10X_1X_3+5X_3X_1+10X_2^2-45X_1^2X_2-34X_1X_2X_1- \\&\quad26X_2X_1^2+105X_1^4.
\end{split}
\end{align}

\paragraph{Noncommutative Dynkin-Faà di Bruno as an incidence Hopf algebra}\label{ncfdb}

We proceed the same way as for the Fa\`a di Bruno Hopf algebra, starting from the family $\mathcal{Q}$ of finite set partition posets made of intervals $[S,T]$, where $S<T$ are partitions of some totally ordered finite set. The equivalence relation $\sim$ will however be finer than the poset isomorphism relation.

The \emph{ordinal sum} $A\sqcup B$ of two totally ordered finite sets $A$ and $B$ is their disjoint union endowed with the unique total order extending both orders of $A$ and $B$, such that $x<y$ for any $x\in A$ and $y\in B$. The ordinal sum is obviously noncommutative. 

We assume a total order on our finite sets, and we order the blocks of a given partition by their maxima. For any partitions $S$ and $S'$ of the totally ordered sets $\{a_1,\ldots ,a_n\}$ and $\{b_1,\ldots ,b_n\}$ respectively, we will write $S\simeq S'$ if and only if there is a bijection $\tau$ from $\{a_1,\ldots ,a_n\}$ onto  $\{b_1,\ldots ,b_n\}$ which bijectively sends any block of $S$ onto a block of $S'$, and which moreover preserves the order of the maxima:
\begin{align}
	\hbox{max }S_1<\hbox{max }S_2\Longleftrightarrow 
	\hbox{max }\varphi(S_1)<\hbox{max }\varphi(S_2)
\end{align}
for any pair $S_1,S_2$ of blocks of $S$. There is a unique such bijection $\tau$ which is increasing when restricted to any block of $S$. We will call it the \emph{canonical permutation} associated to the equivalence relation $\simeq$. Now the equivalence relation on $\mathcal{Q}$ is defined by:
\begin{align}
	[S,T]\sim[S',T'] \hbox{ if and only if } S\simeq S'\hbox{ and } T\simeq T'.
\end{align}
If $S$ and $S'$ are two partitions of $A$ and $B$ respectively, we write $S\sqcup S'$ for the partition of $A\sqcup B$ obtained by concatenation. This ordinal sum of partitions yields a natural identification between set partitions of $A\sqcup B$ and ordered pairs $(S,S')$ where $S$ and $S'$ are set partitions of $A$ and $B$, respectively. The family $\mathcal{Q}$ is hereditary because of the identification of $[S,T]\times[S',T']$ with $[S\sqcup S',T\sqcup T']$. Note that $S\sqcup S'$ is not equivalent to $S'\sqcup S$ in general, hence the cartesian product $[S,T]\times[S',T']$ is not equivalent to $[S',T']\times[S,T]$.  The family $\mathcal{Q}$ is obviously interval-closed.

In order to establish that the associated incidence coalgebra is a Hopf algebra, it remains to show that the equivalence $\sim$ is order-compatible (it is obviously reduced). Let $[S,T]\sim[S',T']$, and let $\tau$ be the canonical bijection associated with the equivalence $T\simeq T'$. For any $U\in[S,T]$, let $\tau(U)$ be the partition with blocks $\tau(U_j)$, where the $U_j$ are the blocks of $U$. We clearly have $U\simeq\tau(U)$, with canonical bijection $\tau$. Hence $[S,U]\sim[S',\tau(U)]$ and $[U,T]\sim[\tau(U),T']$, which proves the assertion.

Let $X_n$ be the equivalence class of the set partition poset of $\{1,\ldots, n+1\}$. The coefficient in front of $X_{r_1}\cdots X_{r_m}\otimes X_k$ is the number of partitions of $\{1,\ldots, n+1\}$ into $k+1$ blocks $P_1,\ldots,P_{k+1}$,  where $P_j$ is of size $r_j+1$ for $j=1,\ldots,m$, and of size one for $j=m+1,\ldots,k+1$, and such that:
\begin{align}
	\max{P_1}<\cdots <\max{P_m}.
\end{align}
Note that we do not care of the max ordering of the one-sized blocks, reflecting the fact that the unit $X_0$ commutes with any other element. In view of Proposition \ref{partitionBell} and Theorem \ref{explicitBell}, the coproduct can be rewritten as:
\begin{align}
	\Delta(X_n)=\sum_{k=0}^n B_{n+1,k+1}(X_0,X_1,X_2,\ldots)\otimes X_k,
\end{align}
where the $B_{n+1,k+1}$ are the noncommutative Bell polynomials. Hence the incidence Hopf algebra described here coincides with the noncommutative Dynkin-Faà di Bruno Hopf algebra. The Dynkin-Faà di Bruno Hopf algebra turns out to be isomorphic to the noncommutative Faà di Bruno Hopf algebra of \cite{brouder2006nha}. To see this, apply the coproduct formulae of Remark 3.8, which still make sense in the noncommutative Dynkin-Faà di Bruno Hopf algebra, and compare with formula (2.16) in \cite{brouder2006nha}\footnote{We thank Jean-Yves Thibon for pointing this out to us.}. See also \cite[Section 6.1]{novelli2008nsf}.

Since the rank polynomials in this incidence Hopf algebra are the noncommutative partial Bell polynomials, Theorem \ref{incidenceAntipode} gives us the following description of the antipode:
\begin{theorem}
The antipode in the noncommutative Dynkin-Faà di Bruno Hopf algebra can be written as
\begin{align} 
	S(X_n) &= \big|\big(B_{n-i+2, n-j+1}(X_0,X_1,X_2, \dots)\big)_{1\leq i,j \leq n}\big|_{1n},
\end{align}
where $B_{n,k}$ are the noncommutative partial Bell polynomials and $|\cdot|_{1n}$ is the quasideterminant computed at the top right element.
\end{theorem}
Example:
\begin{align}
&\left|\begin{array}{cccc} B_{5,4} & B_{5,3} & B_{5,2} & B_{5,1}\\ 1 & B_{4,3} & B_{4,2} & B_{4,1} \\ 0 & 1 & B_{3,2} & B_{3,1} \\ 0 & 0 & 1 & B_{2,1}\end{array}\right|_{1,4} = \\
&\quad -B_{5,1} + B_{5,4}B_{4,1} + B_{5,3}B_{3,1} + B_{5,2}B_{2,1} - B_{5,4}B_{4,3}B_{3,1} - B_{5,4}B_{4,2}B_{2,1} \\
&\quad - B_{5,3}B_{3,2}B_{2,1} + B_{5,4}B_{4,3}B_{3,2}B_{2,1}\\
& \quad = S(X_5)
\end{align}

\subsubsection{Another noncommutative incidence Hopf algebra}

We start with the same hereditary interval-closed family $\mathcal{Q}$ of posets as in Section \ref{ncfdb}. For any partitions $S$ and $S'$ of the totally ordered sets $A:=\{a_1,\ldots ,a_n\}$ and $B:=\{b_1,\ldots ,b_n\}$ respectively, with $a_1<\cdots <a_n$ and $b_1<\cdots <b_n$, we write $S\cong S'$ if and only if the unique increasing bijection $\tau:A\to B$ sends any block of $S$ onto a block of $S'$. Now the equivalence relation on $\mathcal{Q}$ is defined by:
\begin{align}
	[S,T]\approx[S',T'] \hbox{ if and only if } S\cong S'\hbox{ and } T\cong T'.
\end{align}
The order-compatibility of this equivalence relation $\approx$ is obvious, giving rise to an incidence Hopf algebra $\H$. This equivalence relation is finer than the equivalence $\sim$ of Section \ref{ncfdb}, and both are finer than poset isomorphism. Hence we have surjective Hopf algebra morphisms:
\begin{align}
\H\longrightarrow\hskip -5mm\longrightarrow \H_{DFdB}\longrightarrow\hskip -5mm\longrightarrow\H_{FdB}.
\end{align}
A basis of the homogeneous component $\H_n$ is given by intervals $[S,T]$ where $S$ and $T$ are partitions of $\{1,\ldots, ,n+1\}$. One can see that $\H$ is the free associative algebra generated by intervals $[S,T]$ where $T$ is a partition which cannot be written as $T_1\sqcup T_2$ where $T_1$ is a partition of, say, $\{1,\ldots ,r\}$ and $T_2$ is a partition of $\{r+1,\ldots ,n+1\}$.

\subsubsection{Bell polynomials and Möbius inversion}\label{Sect:mobius}
One may ask for formulas expressing the generators $d_i$ of the Bell polynomial in terms of Bell polynomials. Such formulas can be found using \emph{Möbius inversion} in certain variants of the Faà di Bruno and Dynkin Faà di Bruno Hopf algebras. More precisly, we consider situations that are similar to the ones of Sections \ref{Sect:FdB} and \ref{Sect:DFdB}, but we no longer assume that the degree $0$ element is $1$. The constructions are analogous to Möbius inversion on the lattice of partitions (\cite{stanley2011ec1, rota1964otf}).

\paragraph{Commutative case.} We look at the graded commutative algebra $\D = \K\langle \{d_i\} \rangle$ from Section \ref{Sect:noncommBell}, but now without the empty word $\mathbb{I}$, and graded by $|d_i| = i-1$. We add the inverse $d_1^{-1}$ of $d_1$ to the degree $0$ part. The resulting commutative algebra $\H'$ is a bialgebra, graded but not connected, with counit $\epsilon(d_n) = \delta_{n,0}$ and coproduct
\begin{align}
\Delta'(d_n) = \sum_{k=1}^n B_{n,k}(d_1,d_2,\ldots, d_{n})\otimes d_{k}
\end{align}
We get
\begin{align}
\begin{split}
	\Delta'(d_1) &= d_1 \otimes d_1 \\
	\Delta'(d_2) &= d_2 \otimes d_1 + d_1^2 \otimes d_2\\
	\Delta'(d_3) &= d_3 \otimes d_1 + d_1^3 \otimes d_3 + 3d_1d_2 \otimes d_2 \\
	\Delta'(d_4) &= d_4 \otimes d_1 + d_1^4 \otimes d_4 + (3d_2^2 + 4d_1d_3) \otimes d_2 
				+ 6d_1^2d_2 \otimes d_3\\
	\Delta'(d_5) &= d_5\otimes d_1+d_1^5\otimes d_5+(10 d_2d_3+5d_1d_4)\otimes d_2\\ 
			   & \qquad\qquad +(10 d_1^2d_3+15 d_1d_2^2)\otimes d_3+ 10d_1^3d_2\otimes d_4.
\end{split}
\end{align}
Note that $d_1$ is the only group like element. Since the group like elements of the commutative bialgebra $\H'$ are invertible, it is a Hopf algebra (\cite{takeuchi1971fha}).\footnote{Alternatively, the same argument we will use in the noncommutative case can be used also here: it is a Hopf algebra because we can recursively define both a left and a right antipode, which then must coincide.}  The antipode can be given a recursive description by setting $S'(d_1) = d_1^{-1}$, and using either of the defining relations $S' \star \id = \id \star S' = \epsilon$. From $\id \star S' = \epsilon$ we obtain
\begin{align}
	S'(d_n) = d_1^{-n}\left( - d_1^{-1}d_n 
			- \sum_{k=2}^{n-1} B_{n,k}(d_1,d_2,\ldots ,d_{n}) S'(d_{k}) \right).
\end{align}
For example:
\begin{align}
	S'(d_2) &= - d_1^{-3}d_2\\
	S'(d_3) &= - d_1^{-4}d_3 + 3 d^{-5}_1d_2^2\\  
	S'(d_4) &=- d_1^{-5}d_4 + 10d_1^{-6}d_2d_3 - 15d_1^{-7}d_2^3.
\end{align}

The set of linear homomorphisms $\H'^* = \L(\H', \K)$ from $\H'$ to $\K$ are called the \emph{characters} of $\H'$, and they act on the endomorphisms $\End(\H')$ on $\H'$ from the left and the right via the \emph{convolution}:
\begin{align}
	\alpha \star \beta = m_{\H'} \circ (\alpha \star \beta) \circ \Delta'.
\end{align}
The characters form a group, with inverses given by composition with the antipode: $\phi^{-1} = \phi \circ S'$. The \emph{zeta character} $\zeta \in \H'^*$ is defined by $\zeta(d_i)=1$. From the coproduct in $\H'$, we find that the multiplicative map $B:=\id \star \, \zeta$, where $\id$ is the identity endomorphism, evaluated on the $d_i$ gives the Bell polynomials:
\begin{align}
	B(d_i) = \id \star \, \zeta (d_i) = m_{\H'}\circ (\id \otimes \zeta)\circ \Delta'(d_i)  
						= B_{i}(d_1,d_2,\ldots ,d_{i}).
\end{align}
The right antipode $S'$ can be used to invert this equality. We define $\mu:=\zeta \circ S'$ to be the Möbius character, such that $\id=B \star \mu$, and obtain
\begin{align}
	d_i = B \star \mu (d_i) = m_{\H'}\circ (B \otimes \mu)\circ \Delta'(d_i).
\end{align}
For example:
\begin{align}
	d_1 &= B(d_1) = B_1(d_1) \\
	d_2 &= B(d_2) - B(d_1^2)=B_2(d_1,d_2) - B_1(d_1)B_1(d_1)\\
	d_3 &= B(d_3) + 2B(d_1^3) - 3B(d_1)B(d_2)=B_3 - 3B_1B_2 + 2 B_1^3
\end{align}

\paragraph{Noncommutative case.} Similar constructions can be done also in the noncommutative setting of Section \ref{Sect:DFdB}. We assume $d_1$ is not idempotent, and that $d_1^{-1}$ exists. As before, the resulting noncommutative bialgebra $\tilde{\H} = \mathcal{D}$ is graded by $|d_i| = i-1$, but not connected. The counit is $\epsilon(d_n) = \delta_{n,0}$, and the coproduct is defined in terms of the noncommutative Bell polynomials:
\begin{align}
\tilde\Delta (d_n) = \sum_{k=1}^n B_{n,k}(d_1,d_2,\ldots, d_{n})\otimes d_{k}
\end{align}
For example,
\begin{align}
\begin{split}
	\tilde\Delta(d_1) &= d_1 \otimes d_1 \\
	\tilde\Delta(d_2) &= d_2 \otimes d_1 + d_1^2 \otimes d_2\\
	\tilde\Delta(d_3) &= d_3 \otimes d_1 + d_1^3 \otimes d_3 
				+ (2d_1d_2 + d_2d_1) \otimes d_2 \\
	\tilde\Delta(d_4) &= d_4 \otimes d_1 + d_1^4 \otimes d_4 
				+ (3d_2^2 + d_3d_1 + 3d_1d_3) \otimes d_2 \\
			      & \hspace{3.5cm}	+ (3d_1^2d_2 + d_2d_1^2 + 2d_1d_2d_1)\otimes d_3
\end{split}
\end{align}

We can define the antipode $\tilde{S}$ on $\tilde{H}$ by setting $\tilde{S}(d_1) = d_1^{-1}$ and using $\id \star \tilde{S} = \epsilon$ to obtain a recursion
\begin{align}
	\tilde{S}(d_n) = d_1^{-n}\left( - d_nd_1^{-1} 
	- \sum_{k=2}^{n-1} B_{n,k}(d_1,d_2,\ldots ,d_n) \tilde{S}(d_k) \right).
\end{align}
For example:
\begin{align}
	\tilde{S}(d_2) &= - d_1^{-2}d_2d_1^{-1}\\
	\tilde{S}(d_3) &= - d_1^{-3}d_3d_1^{-1} + 2d_1^{-2}d_2d_1^{-2}d_2d_1^{-1} 
			+ d_1^{-3}d_2d_1^{-1}d_2d_1^{-1}
\end{align}
In addition, using $\tilde{S} \star id= \epsilon$, we get\footnote{Note that the two formulas give the same $S$, because, writing $S'$ for the antipode defined by $S' \star \id = \epsilon$ (i.e. the left antipode), we have $S = S \star \epsilon = S \star (\id \star S') = (S \star \id) \star S' = S'$.}
\begin{align}
\tilde{S}(d_n) = \left( - d_1^{-n}d_n - \sum_{k=2}^{n-1} \tilde{S}(B_{n,k}(d_1,d_2,\ldots, d_{n})) d_k \right)d_1^{-1}.
\end{align}

We again consider characters in $\tilde{\H}$, which act on the endomorphisms via convolution. We find that the multiplicative map $B := \id \star \, \zeta$ evaluated on the $d_i$ results in the noncommutative Bell polynomials:
\begin{align}
	B(d_i) = \id \star \, \zeta (d_i) = m_\mathcal{D}\circ (d \otimes \zeta)\circ\Delta'(d_i)  
		= B_{i}(d_1,d_2,\ldots ,d_{i}).
\end{align}
Composition of $\zeta$ with the antipode $\tilde{S}$ gives the Möbius character $\mu:=\zeta \circ \tilde{S}$ with values in $\K$, satisfying $\zeta \star \mu = \mu \star \zeta=\epsilon$. We obtain $d=B \star \mu$, so
\begin{align}
d_i= B \star \mu (d_i) = m_\mathcal{D}\circ(B \otimes \mu)\circ \Delta'(d_i).
\end{align}
For example:
\begin{eqnarray}
	d_1 &=& B(d_1) = B_1(d_1) \\
	d_2 &=& B(d_2) - B(d_1^2)=B_2(d_1,d_2) - B_1(d_1)^2\\
	d_3 &=& B(d_3) + 2B(d_1^3) - 2B(d_1)B(d_2)  - B(d_2)B(d_1)\\
	        &=& B_3(d_1,d_2,d_3) - 2B_1(d_1)B_2(d_1,d_2) - B_2(d_1,d_2)B_1(d_1) + 2 B_1(d_1)^3
\end{eqnarray}

\section*{Acknowledgements}
We wish to thank Henning Lohne and Hans Munthe-Kaas for useful discussions. K.E.-F. is supported by a Ramón y Cajal research grant from the Spanish government. A.L. was supported by an ERCIM ``Alain Bensoussan Fellowship'', funded by the European Union Seventh Framework Programme (FP7/2007-2013) under grant agreement no. 246016. D.M. and K.E.-F. were supported by Agence Nationale de la Recherche (projet CARMA).

\bibliography{bibliography}

\end{document}